\documentclass{amsart}
\makeatletter
\@namedef{subjclassname@2020}{%
\textup{2020} Mathematics Subject Classification}
\makeatother
\usepackage{geometry}
\usepackage{amsthm}
\usepackage{setspace}
\usepackage{mathtools}
\usepackage{subcaption}
\usepackage{mathpazo} 
\usepackage[colorlinks,allcolors=blue]{hyperref} 
\usepackage{xpatch} 
\xpatchcmd{\proof} 
{\itshape} 
{\bfseries} 
{}
{}
\newtheorem{theorem}{Theorem}
\newtheorem{proposition}{Proposition}
\newtheorem{lemma}{Lemma}
\newtheorem{corollary}{Corollary}
\newtheorem{definition}{Definition}
\theoremstyle{remark}
\newtheorem{remark}{Remark}

\newcommand{\C}{\mathbb{C}}

\newcommand{\T}{\mathbb{T}}
\newcommand{\B}{\mathbb{B}} 
\newcommand{\D}{\Omega}

\allowdisplaybreaks
\DeclareCaptionLabelFormat{fullfig}{Figure~\thefigure.#2}
\onehalfspace

\title{Convexity of the Berezin range of composition operators induced by automorphisms of the Unit Ball in $\C^{n}$}

\author{Timothy G. Clos}
\author{Trevor C. Pentzien}
\author{Tomas Miguel Rodriguez}

\address{Trevor C. Pentzien, The University of Toledo, Department of 
	Mathematics \& Statistics, 2801 W. Bancroft, Toledo, OH 43606, USA}
\email{trevor.pentzien@rockets.utoledo.edu}

\address{Timothy G. Clos, Kent State University, Department of 
	Mathematical Sciences, 800 E Summit St, Kent, OH 44240, USA}
\email{tclos@kent.edu}

\address{Tomas Miguel Rodriguez, Temple University Japan, Department of 
	Business and Technology, 1-14-29 Taishido, Setagaya-ku, Tokyo 154-0004, Japan}
\email{tomi.rodriguez1@tuj.temple.edu}

\date{\today}
\begin{document}

\begin{abstract}
We study the convexity of the Bergman space Berezin range of composition operators induced by unit disk and unit ball automorphisms.  We obtain several new results characterizing convexity of the Berezin range for linear fractional unit disk and unit ball automorphisms composed with unitary maps. 
\end{abstract}

\maketitle

Let $\Omega$ be a domain in $\mathbb{C}^n$ and assume that $\varphi: \Omega \rightarrow \Omega$ is holomorphic. Then the composition operator with symbol $\varphi$, acting on the space of all holomorphic functions on $\Omega$, is defined by 
\[C_{\varphi}f(z)=f \circ\varphi(z)\]
for $z \in \Omega$. Let $A^{2}(\Omega)$  denote the Bergman space over the domain $\Omega$ in $\C^{n}$ which is the space of all holomorphic functions $f: \Omega \rightarrow \mathbb{C}$ such that 
\[\int_{\Omega}|f(z)|^2dV(z)< \infty\]
where $dV$ denotes the volume measure on $\Omega$.  It is known that point evaluation is a linear and bounded functional on the Bergman space on any bounded domain $\Omega \subset \mathbb{C}^n$ for $n \geq 1$. Therefore, by the Riesz representation theorem, for each fixed $z \in \Omega$ there exists a function $K_z \in A^2(\Omega)$ so that $f(z)=\langle f,K_z \rangle$ for any $f \in A^2(\Omega)$.  The function $K_z$ is called the Bergman kernel. Further we define, 
\[k_z(w)=\dfrac{K_z(w)}{\| K_z\|}=\dfrac{K_z(w)}{\sqrt{K_z(z)}},\]
for $w,z \in \Omega$  as the normalized Bergman kernel. Lastly, we define the Berezin transform of a bounded linear operator $T$ on $A^2(\Omega)$ at $z \in \D$  as the function
\[\widetilde{T}(z)= \langle Tk_z,k_z \rangle.\]
Thus the Berezin range of $T$ is 
\[B(T)=\left\{\langle Tk_z,k_z \rangle: z \in \Omega\right\}.\]

In most of the literature the Berezin transform has been studied from a function theoretic perspective. That is, the Berezin transform is a useful tool for characterizing invertibility and compactness of Toeplitz operators on Bergman spaces.  For a study on compactness of Toeplitz operators via Berezin transforms on the unit disk, see \cite{AxlerZheng}.  For higher dimensional analogs of the results in \cite{AxlerZheng}, see \cite{Coburn73, Eng99, CSZ, DS25}.
The Berezin range of an operator $T$ is a subset of the numerical range of $T$, 
\[W(T)=\left\{\langle Tf,f \rangle: f\in A^2(\Omega)\,,\|f\| =1\right\}.\]
By the Toeplitz-Hausdorff theorem \cite{T1918,H1919}, the numerical range of an operator T is always a convex set. Thus, we are interested in studying the Bergman space Berezin transform from a more geometric perspective.  In particular, we study the convexity of Bergman space Berezin range of a composition operator $B(C_{\varphi})$ with an automorphism symbol $\varphi$ on the unit disk $\mathbb{D}$ in $\mathbb{C}$ and the unit ball $\mathbb{B}_n$ in $\mathbb{C}^n$ for $n\geq 2$.

We define 
\[\mathbb{T}=\{\xi \in \mathbb{C}: |\xi|=1\},\] 
and for $a=(a_1,\ldots,a_n)\in \mathbb{C}^n$, \[\|a\|^2=\sum_{j=1}^n |a_j|^2.\]
In $\C$, a map $\varphi_{a,\xi}:\mathbb{D} \rightarrow \mathbb{D}$ is an automorphism if there are $\xi \in \mathbb{T}$ and $a \in \mathbb{D}$ such that 
\[\varphi_{a,\xi}(z)=\xi\dfrac{a-z}{1-\overline{a}z}.\]
On the other hand for $n > 1$, a map $\varphi_{a,U}: \mathbb{B}_{n} \to \mathbb{B}_{n}$ is an automorphism if there are $a \in \mathbb{B}_{n}$, and unitary map $U$ such that,
\[\varphi_{a,U}(z)= \begin{cases}U\circ\left(\frac{a-P_a(z)-s_aQ_a(z)}{1-\langle z,a\rangle}\right), &a \neq 0 \\ U(z), &a = 0\end{cases},\]
where $s_a=(1-\|a\|^2)^{\frac{1}{2}}$, $P_a(z)=\frac{\langle z,a\rangle}{\|a\|^2}a,$ and $Q_a=I-P_a$. We will use the notation 
\[\varphi_{a,I}(z)=\frac{a-P_a(z)-s_aQ_a(z)}{1-\langle z,a\rangle}\] and 
\[\varphi_{a,-I}(z)=-\left(\frac{a-P_a(z)-s_aQ_a(z)}{1-\langle z,a\rangle}\right).\]

For a composition operator $C_{\varphi_{a,\xi}}: H^2(\mathbb{D})\rightarrow H^2(\mathbb{D})$ with an automorphism symbol $\varphi_{a,-1}(z)=\dfrac{z-a}{1-\overline{a}z}$, \cite{cowen2022convexity} showed that $B(C_{\varphi_{a,-1}})$ is convex if and only if $a=0$.  Here, $H^2(\mathbb{D})$ is the Hardy space of the unit disk.  As alluded to in \cite{cowen2022convexity}, the Bergman space case (as opposed to the Hardy space) is not as well understood because of the added complexity of the Bergman space.  Despite this, \cite{augustine2023composition} proved that for the same class of symbols, as studied in \cite{cowen2022convexity}, the result generalizes to the Bergman space over the disk. See also \cite{AUGUSTINE2026103762} and \cite{AS} for surveys of results.  We studied the convexity of the Bergman space Berezin range in the disk when the automorphism has parameters $a \in \mathbb{D}$, and $\xi = 1$ or $\xi = \pm i$. Further we study the analogous problem in the unit ball in $\C^{n}$, when the automorphism has parameters $a \in \mathbb{B}_{n}$, and $U=I$ or $U=-I$ where $I$ is the identity map.  It is understood that the Berezin range refers to the Bergman space Berezin range.  

\section{Main Results}
The following theorems are the main results.  These results on the disk only consider $\xi \in \{1,i,-i\}$. 
\begin{theorem}\label{MainResultOneDimension}
The Berezin range of $C_{\varphi_{a,1}}$ on $A^2(\mathbb{D})$ is convex if and only if $0\leq|a| \leq \frac{\sqrt{2}}{2}$. 
\end{theorem}
\begin{theorem}\label{mainresultonedimension3}
The Berezin range of 
$C_{\varphi_{a,i}}$ or $C_{\varphi_{a,-i}}$ on $A^2(\mathbb{D})$ is convex if and only if $a=0$.
\end{theorem}
See \cite[Theorem 5.7]{augustine2023composition} for the characterization of convexity of the Berezin range for the operator $C_{\varphi_{a,-1}}$ on $A^2(\mathbb{D})$.  This is analogous to the result in \cite[Theorem 4.5]{cowen2022convexity}.  It is similar to our result for $\xi=\pm i$ but different for $\xi=1$.  For $n\geq 2$ we have the following.
\begin{theorem}\label{Zerotonegative}
The Berezin range $B(C_{\varphi_{a,-I}})$ on $A^{2}(\B_{n})$ for $n\geq 2$ is convex if and only if $a = 0.$
\end{theorem}
\begin{theorem}\label{Zerotopositive}
If the Berezin range of $C_{\varphi_{a,I}}$ on $A^{2}(\B_{n})$ for $n\geq 2$ is convex, then \[\|a\|\leq \tan\left(\frac{\pi}{2(n+1)}\right). \]       
\end{theorem}
Notice Theorem \ref{Zerotopositive} implies there exists $a\in \mathbb{B}_n$ for $n\geq 2$ so that the Berezin range of $C_{\varphi_{a,I}}$ is not convex. By Theorem \ref{MainResultOneDimension}, this is also true for $n=1$.
\section{\texorpdfstring{The Unit Disk in $\mathbb{C}$}{The Unit Disk in ℂ}}
We first compute the Berezin transform for a general symbol $\varphi:\Omega \to \Omega$. We do this for the Bergman space, but the computation is identical for any reproducing kernel Hilbert space.  We note that $K_z(z)>0$ for any $z\in \Omega$.
\begin{lemma} \label{BerezinRangeComposition}
Let $\Omega$ be a domain and let $K$ be the kernel for the Bergman space over $\Omega$. Furthermore, let $C_{\varphi}:A^2(\Omega) \rightarrow A^2(\Omega)$ be the composition operator with symbol $\varphi$.
Then for $z \in \Omega$,
\begin{equation*}
    \widetilde{C}_{\varphi}(z) = \frac{C_{\varphi}K_{z}(z)}{K_{z}(z)}.
\end{equation*}
\begin{proof}
Let $z \in \Omega$, then  
\begin{align*}
    \widetilde{C}_{\varphi}(z) &= \langle C_{\varphi}k_z, k_{z} \rangle \\
                        &= \left\langle C_{\varphi}\frac{K_{z}}{\|K_{z}\|}, \frac{K_{z}}{\|K_{z}\|} \right\rangle \\
                        &= \left\langle \frac{K_{z} \circ \varphi}{\|K_{z}\|}, \frac{K_{z}}{\|K_{z}\|} \right\rangle \\
                        &= \frac{1}{\|K_{z}\|^{2}}\left\langle K_{z} \circ \varphi, K_{z} \right\rangle \\
                        &= \frac{1}{K_{z}(z)} K_{z}(\varphi(z)) \\
                        &= \frac{(C_{\varphi}K_{z})(z)}{K_{z}(z)}.
\end{align*}
\end{proof}
\end{lemma}
\noindent
We can use Lemma \ref{BerezinRangeComposition} to find a formula for the Berezin transform of a composition operator in the Bergman space over the unit disk with symbol $\varphi_{a,\xi}(z)=\xi\dfrac{a-z}{1-\overline{a}z}$.
\begin{lemma}\label{ReImBerezinTransform}
The real and imaginary parts of $\widetilde{C}_{\varphi_{a,\xi}}$ are of the form,
\[\mathrm{Re}(\widetilde{C}_{\varphi_{a,\xi}})(z)=\left[\dfrac{1-|z|^2}{A(z)^2+B(z)^2}\right]^2\left(C(z)[A(z)^2-B(z)^2]+4A(z)B(z)D(z)\right),\]
\[\mathrm{Im}(\widetilde{C}_{\varphi_{a,\xi}})(z)=2 \left[\dfrac{1-|z|^2}{A(z)^2+B(z)^2}\right]^2(D(z)[A(z)^2-B(z)^2]-A(z)B(z)C(z)),\]
where
\[A(z)= -\mathrm{Re}(\overline{a}z)(1+\mathrm{Re}(\xi))+ 1 + \mathrm{Re}(\xi)|z|^2-\mathrm{Im}(\xi)\mathrm{Im}(\overline{a}z),\]
\[B(z)=-\mathrm{Im}(\overline{a}z)(1-\mathrm{Re}(\xi))+\mathrm{Im}(\xi)|z|^2-\mathrm{Im}(\xi)\mathrm{Re}(\overline{a}z),\]
\[C(z)=(1-\mathrm{Re}(\overline{a}z))^2-\mathrm{Im}(\overline{a}z)^2,\]
\[D(z)=-(1-\mathrm{Re}(\overline{a}z))\mathrm{Im}(\overline{a}z).\]
\end{lemma}    
\begin{proof}
Let $\xi \in \mathbb{T}$. Then, by Lemma \ref{BerezinRangeComposition}

\begin{align*}
\widetilde{C}_{\varphi_{a,\xi}}(z) =& \left[\dfrac{1-|z|^2}{1-\xi(\frac{a-z}{1-\overline{a}z})\overline{z}}\right]^2 \\
=& \left[\dfrac{(1-|z|^2)(1-\overline{a}z)}{1-\overline{a}z-\xi a\overline{z}+\xi|z|^2}\right]^2 \\
=& \left[\dfrac{(1-|z|^2)(1-\overline{a}z)}{A(z)+iB(z)}\right]^2 \\
=& \left[\dfrac{(1-|z|^2)(1-\overline{a}z)(A(z)-iB(z))}{A(z)^2+B(z)^2}\right]^2 \\
=& \left[\dfrac{1-|z|^2}{A(z)^2+B(z)^2}\right]^2(1-\overline{a}z)^2(A(z)-iB(z))^2 \\
=& \left[\dfrac{1-|z|^2}{A(z)^2+B(z)^2}\right]^2 (\mathrm{Re}(1-\overline{a}z)+i\mathrm{Im}(1-\overline{a}z))^2(A(z)-iB(z))^2 \\
=& \left[\dfrac{1-|z|^2}{A(z)^2+B(z)^2}\right]^2(\mathrm{Re}(1-\overline{a}z)^2-\mathrm{Im}(1-\overline{a}z)^2+2i\mathrm{Re}(1-\overline{a}z)\mathrm{Im}(1-\overline{a}z))(A(z)^2-B(z)^2-2iA(z)B(z)) \\
=& \left[\dfrac{1-|z|^2}{A(z)^2+B(z)^2}\right]^2(C(z)+2iD(z))(A(z)^2-B(z)^2-2iA(z)B(z)) \\
=& \left[\dfrac{1-|z|^2}{A(z)^2+B(z)^2}\right]^2[(C(z)[A(z)^2-B(z)^2]+4A(z)B(z)D(z)) + i(2D(z)[A(z)^2-B(z)^2]-2A(z)B(z)C(z))] \\
\end{align*}
\end{proof}
Now that we have found the form for $\widetilde{C}_{\varphi_{a,\xi}}$ in the Bergman space over the unit disk, we can use a computer software to get an insight for some of the properties of the Berezin range for different values of $a \in \mathbb{D}$ and $\xi \in \mathbb{T}$. In the following figures the colored bar on the right side of the figures shows the distance from the origin for the input value $z \in \mathbb{D}$.
\begin{figure}[htbp]
    \centering

    \begin{subfigure}{0.40\textwidth}
        \centering
    \includegraphics[width=\linewidth]{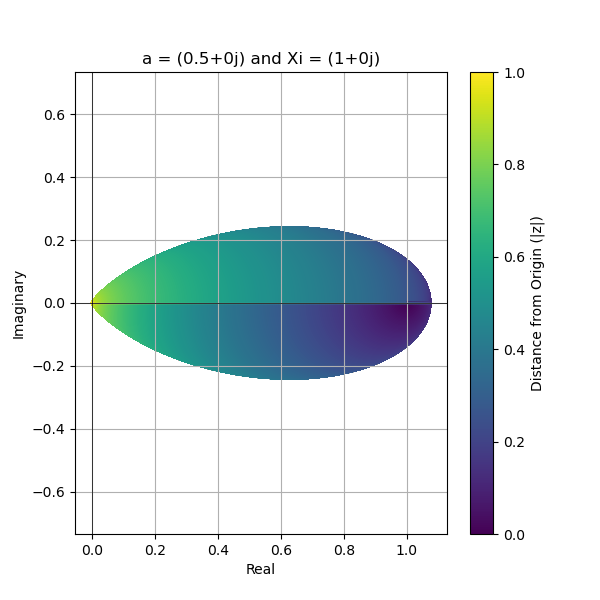}
    \caption{Berezin range $B(C_{\varphi_{0.5,1}})$ in the Bergman space over the unit disk $\mathbb{D}$.}
    \end{subfigure}
    \hfill
    \begin{subfigure}{0.40\textwidth}
        \centering
        \includegraphics[width=\linewidth]{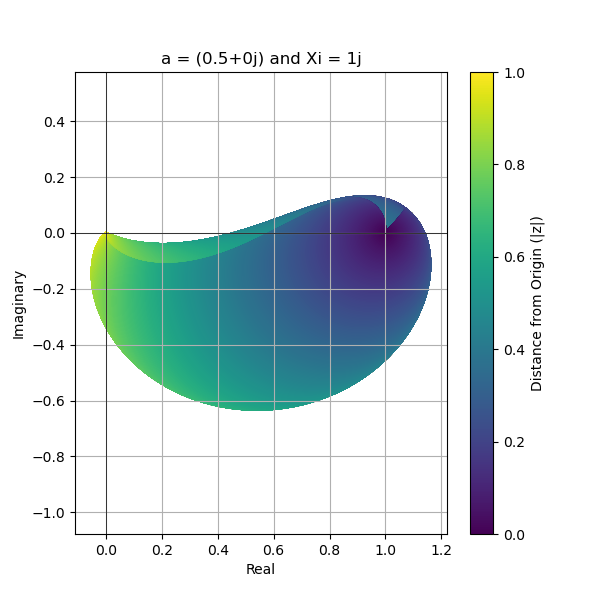}
        \caption{Berezin range $B(C_{\varphi_{0.5,i}})$ in the Bergman space over the unit disk $\mathbb{D}$.}
    \end{subfigure}
\end{figure}
\newpage
\begin{figure}[htbp]
    \ContinuedFloat
    \vspace{1cm}

    \begin{subfigure}{0.40\textwidth}
        \centering
        \includegraphics[width=\linewidth]{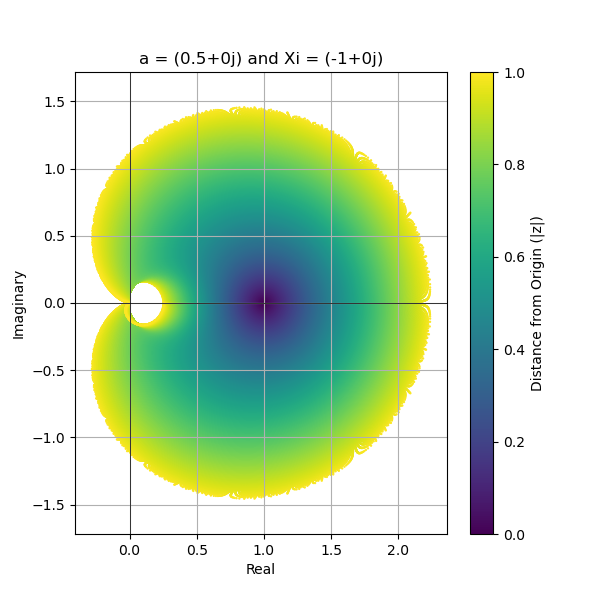}
        \caption{Berezin range $B(C_{\varphi_{0.5,-1}})$ in the Bergman space over the unit disk $\mathbb{D}$.}
    \end{subfigure}
    \hfill
    \begin{subfigure}{0.40\textwidth}
        \centering
        \includegraphics[width=\linewidth]{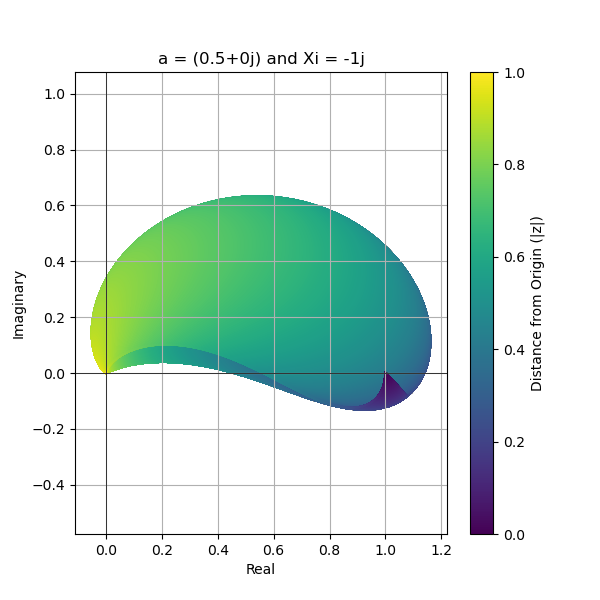}
        \caption{Berezin range $B(C_{\varphi_{0.5,-i}})$ in the Bergman space over the unit disk $\mathbb{D}$.}
    \end{subfigure}
\end{figure}
One of the properties of interest is for what values of $a \in \mathbb{D}$ and $\xi \in \T$ the Berezin range is symmetric about the real line.
\begin{proposition}\label{conjugateberezin}
$B(C_{\varphi_{a,\xi}})=\overline{B(C_{\varphi_{a,\overline{\xi}}})}$    
\end{proposition}

\begin{proof}
Let $\xi \in \mathbb{T}$, $z=re^{i\theta}$, and nonzero $a=\rho e^{i\psi}$. We want to show that there exists $w \in \mathbb{D}$ such that 
\[\widetilde{C}_{\varphi_{a,\xi}}(z)= \overline{\widetilde{C}_{\varphi_{a,\overline{\xi}}}(w)}.\]
For $w=re^{i(2\psi-\theta)} \in \mathbb{D}$, we have $w\overline{a}=a\overline{z}$ which is equivalent to $\mathrm{Re}(\overline{a}z)=\mathrm{Re}(\overline{a}w)$ and $\mathrm{Im}(\overline{a}z)=-\mathrm{Im}(\overline{a}w)$.
From Lemma \ref{ReImBerezinTransform}, we have 
\begin{align*}
A_{\xi}(z) &= -\mathrm{Re}(\overline{a}z)(1+\mathrm{Re}(\xi))+ 1 + \mathrm{Re}(\xi)|z|^2-\mathrm{Im}(\xi)\mathrm{Im}(\overline{a}z) \\
&= -\mathrm{Re}(\overline{a}w)(1+\mathrm{Re}(\overline{\xi}))+ 1 + \mathrm{Re}(\overline{\xi})|w|^2-\mathrm{Im}(\overline{\xi})\mathrm{Im}(\overline{a}w) \\
&= A_{\overline{\xi}}(w), \\ \\
B_{\xi}(z) &= -\mathrm{Im}(\overline{a}z)(1-\mathrm{Re}(\xi))+\mathrm{Im}(\xi)|z|^2-\mathrm{Im}(\xi)\mathrm{Re}(\overline{a}z) \\
&= \mathrm{Im}(\overline{a}w)(1-\mathrm{Re}(\overline{\xi}))-\mathrm{Im}(\overline{\xi})|w|^2+\mathrm{Im}(\overline{\xi})\mathrm{Re}(\overline{a}w) \\
&= - B_{\overline{\xi}}(w), \\ \\
C(z) &=\mathrm{Re}(1-\overline{a}z)^2-\mathrm{Im}(1-\overline{a}z)^2 \\
&=\mathrm{Re}(1-\overline{a}w)^2-\mathrm{Im}(1-\overline{a}w)^2 \\
&= C(w), \\ \\
D(z) &= -\mathrm{Re}(1-\overline{a}z)\mathrm{Im}(1-\overline{a}z) \\
&= \mathrm{Re}(1-\overline{a}w)\mathrm{Im}(1-\overline{a}w)\\
&= -D(w).
\end{align*}
By Lemma \ref{ReImBerezinTransform},
\begin{align*}
\overline{\widetilde{C}_{\varphi_{a,\overline{\xi}}}(w)} =& \left[\dfrac{1-|w|^2}{A_{\overline{\xi}}(w)^2+B_{\overline{\xi}}(w)^2}\right]^2(C(w)[A_{\overline{\xi}}(w)^2-B_{\overline{\xi}}(w)^2]+4A_{\overline{\xi}}(w)B_{\overline{\xi}}(w)D(w))\\
& - 2i\left[\dfrac{1-|w|^2}{A_{\overline{\xi}}(w)^2+B_{\overline{\xi}}(w)^2}\right]^2(D(w)[A_{\overline{\xi}}(w)^2-B_{\overline{\xi}}(w)^2]-A_{\overline{\xi}}(w)B_{\overline{\xi}}(w)C(w)) \\ 
=& \left[\dfrac{1-|z|^2}{A_{\xi}(z)^2+B_{\xi}(z)^2}\right]^2(C(z)[A_{\xi}(z)^2-B_{\xi}(z)^2]+4A_{\xi}(z)B_{\xi}(z)D(z))\\
& + 2i\left[\dfrac{1-|z|^2}{A_{\xi}(z)^2+B_{\xi}(z)^2}\right]^2(D(z)[A_{\xi}(z)^2-B_{\xi}(z)^2]-A_{\xi}(z)B_{\xi}(z)C(z)) \\ 
=& \ \widetilde{C}_{\varphi_{a,\xi}}(z).
\end{align*}
This holds for all $z \in \mathbb{D}$, which tells us that $B(C_{\varphi_{a,\xi}})=\overline{B(C_{\varphi_{a,\overline{\xi}}})}$ as sets.  
\end{proof}
This result further shows that the Berezin range is only complex conjugate symmetric when $\xi = \pm 1$. The precise statement is the following corollary.
\begin{corollary}\label{propconj}
Let $\xi \in \mathbb{T}$.  $B(C_{\varphi_{a,\xi}})$ is closed under complex conjugation if and only if $\mathrm{Im}(\xi) = 0$.
\end{corollary}
This is not the only useful property that we obtain from Proposition \ref{conjugateberezin}. If we recall that complex conjugation is convex invariant we get the following,
\begin{corollary}\label{propconvex}
    $B(C_{\varphi_{a,\xi}})$ is convex if and only if $B(C_{\varphi_{a,\overline{\xi}}})$ is convex.
\end{corollary}
Next we will show that the Berezin range is rotationally invariant.  That is, for any fixed unimodular constant $\lambda$, the Berezin ranges corresponding to $\lambda a$ and $a$ are equal.
\begin{proposition}\label{auto_on_disk}
    Let $\tau$ be a rotation of $\mathbb{D}$.  Then the Berezin range $B(C_{\varphi_{\tau(a),\xi}})=B(C_{\varphi_{a,\xi}})$ as sets. 
\end{proposition}
\begin{proof}
Let $\tau(z)=\lambda z$ for some $|\lambda|=1$ and $\tau^*(z)=\overline{\lambda}z$. From Lemma \ref{ReImBerezinTransform},
\[\widetilde{C}_{\varphi_{\tau(a),\xi}}(z)=\left[\dfrac{(1-|z|^2)(1-\overline{\lambda a}z)}{1-\overline{\lambda a}z-\xi \lambda a\overline{z}+\xi|z|^2}\right]^2=\widetilde{C}_{\varphi_{a,\xi}}(\overline{\lambda}z)=\widetilde{C}_{\varphi_{a,\xi}}(\tau^*(z)).\]
Now $\tau^*:\mathbb{D}\to \mathbb{D} $ is surjective, so 
\[\widetilde{C}_{\varphi_{\tau(a),\xi}}(\mathbb{D})=\widetilde{C}_{\varphi_{a,\xi}}(\tau^*(\mathbb{D}))=\widetilde{C}_{\varphi_{a,\xi}}(\mathbb{D}).\]  Thus $B(C_{\varphi_{\tau(a),\xi}})=B(C_{\varphi_{a,\xi}})$ as sets.
This completes the proof.
\end{proof}
\begin{corollary}\label{auto_on_disk2}
    Let $\tau$ be a rotation of $\mathbb{D}$.  Then the Berezin range $B(C_{\varphi_{\tau(a),\xi}})$ on $A^2(\mathbb{D})$ is convex if and only if the Berezin range $B(C_{\varphi_{a,\xi}})$ on $A^2(\mathbb{D})$ is convex.
\end{corollary}
We are ready to prove Theorem \ref{MainResultOneDimension}. 
\begin{proof}[Proof of Theorem \ref{MainResultOneDimension}]
By Proposition \ref{auto_on_disk} without loss of generality we can assume that $a$ is a nonnegative real number.
Suppose $B(C_{\varphi_{a,1}})$ is convex, we want to show that $0\leq a \leq \frac{\sqrt{2}}{2}$.
Assume the contrary. Using Lemma \ref{ReImBerezinTransform} and writing the real and imaginary parts of $\widetilde{C}_{\varphi_{a,1}}(z)$ in terms of polar coordinates we have,
\begin{align} 
\mathrm{Re}(\widetilde{C}_{\varphi_{a,1}})(r,\theta)=&\left(\dfrac{1-r^2}{1-2ar\cos\theta+r^2} \right)^2\left(\left(1-ar\cos\theta \right)^2-\left(ar\sin\theta\right)^2 \right) \label{polarform}  \\
\mathrm{Im}(\widetilde{C}_{\varphi_{a,1}})(r,\theta)=&2\left(\dfrac{1-r^2}{1-2ar\cos\theta+r^2} \right)^2\left(1-ar\cos\theta \right)\left(-ar\sin\theta\right). \nonumber
\end{align}
Suppose $\mathrm{Re}(\widetilde{C}_{\varphi_{a,1}})(r,\theta)=0$, then this is equivalent to $\left(1-ar\cos\theta \right)^2-\left(ar\sin\theta\right)^2 =0$. Further, 
\begin{align}
 \left(1-ar\cos\theta \right)^2-\left(ar\sin\theta\right)^2=0 \quad 
 \Leftrightarrow& \quad 1-2ar\cos\theta+a^2r^2(2\cos^2\theta-1)=0 \nonumber \\
 \Leftrightarrow& \quad 2a^2r^2\cos^2\theta-2ar\cos\theta+1-a^2r^2=0 \nonumber \\
 \Leftrightarrow& \quad \cos^2\theta-\dfrac{1}{ar}\cos\theta+\dfrac{1-a^2r^2}{2a^2r^2}=0. \label{quadraticequation}
\end{align}
The quadratic equation in terms of $\cos \theta$ has a solution for $ \frac{\sqrt{2}}{2r}\leq a.$ By our assumption that $\frac{\sqrt{2}}{2}<a<1$, we have $\frac{\sqrt{2}}{2}<r<1$. We will show that equation \eqref{quadraticequation} has a solution for all $\theta \in [0,2\pi]$ where $\theta$ depends on $a$ and $r$ satisfying the inequality $\frac{\sqrt{2}}{2r} \leq a < 1.$ Solving for the solution of $\cos \theta$, we have 
\begin{align*}
\cos \theta = \dfrac{1 \pm \sqrt{2a^2r^2-1}}{2ar} \Rightarrow& -1 \leq \dfrac{1 \pm \sqrt{2a^2r^2-1}}{2ar} \leq 1 \\
\Rightarrow& -2ar-1 \leq \pm \sqrt{2a^2r^2-1} \leq 2ar-1.
\end{align*}
This indicates that both inequalities  
\begin{align}\label{inequality1}
-2ar-1 \leq \sqrt{2a^2r^2-1} \leq 2ar-1, \text{and}    
\end{align}
\begin{align}\label{inequality2}
-2ar-1 \leq - \sqrt{2a^2r^2-1} \leq 2ar-1    
\end{align}
must hold true.
\noindent In inequality \eqref{inequality1}, $-2ar-1 \leq \sqrt{2a^2r^2-1}$ holds true. Since $a^2r^2-2ar+1=(ar-1)^2 \geq 0$, we have
\begin{align*}
0 \leq a^2r^2-2ar+1  \Leftrightarrow&\  0 \leq 2a^2r^2-4ar+2  \\
\Leftrightarrow&\ 2a^2r^2-1 \leq 4a^2r^2-4ar+1  \\
\Leftrightarrow&\ 2a^2r^2-1 \leq (2ar-1)^2 \\
\Leftrightarrow&\ \sqrt{2a^2r^2-1} \leq 2ar-1. \\
\end{align*}
In inequality \eqref{inequality2}, since $2ar-1 \geq 0$ we have $- \sqrt{2a^2r^2-1} \leq 2ar-1$. Since $a^2r^2+2ar+1=(ar+1)^2 \geq 0$, we have
\begin{align*}
0 \leq a^2r^2+2ar+1 \Leftrightarrow&\ 0 \leq 2a^2r^2+4ar+2 \\
\Leftrightarrow&\ 2a^2r^2-1 \leq 4a^2r^2+4ar+1  \\
\Leftrightarrow&\ 2a^2r^2-1 \leq (-2ar-1)^2  \\
\Leftrightarrow&\ -\sqrt{2a^2r^2-1} \leq -2ar-1. \\
\end{align*}
The affirmation of inequalities \eqref{inequality1} and \eqref{inequality2} shows that for $\frac{\sqrt{2}}{2r} \leq a < 1$ there exists a $\theta(a,r)$, that depends on $a$ and $r$, satisfying equation \eqref{quadraticequation}. Thus, we have $\mathrm{Re}(\widetilde{C}_{\varphi_{a,1}})(r,\theta(a,r))=0$ and $\mathrm{Im}(\widetilde{C}_{\varphi_{a,1}})(r,\theta(a,r)) \neq 0$, showing that
$\widetilde{C}_{{\varphi}_{a,1}}(r,\theta(a,r))\neq 0$ lies on the imaginary axis. By Corollary \ref{propconj}, $-\widetilde{C}_{{\varphi}_{a,1}}(r,\theta(a,r))$ is also in the Berezin range.  By the assumption that $B(C_{\varphi_{a,1}})$ is convex, there is a line segment connecting $\widetilde{C}_{{\varphi}_{a,1}}(r,\theta(a,r))$ and $-\widetilde{C}_{{\varphi}_{a,1}}(r,\theta(a,r))$ is contained in the Berezin range which indicates that $0$ is in the Berezin range as well. This is a contradiction since if $\widetilde{C}_{{\varphi}_{a,1}}(z)=0$ then $|z|=1$. Therefore, if $B(C_{\varphi_{a,1}})$ is convex, then $0 \leq a \leq \frac{\sqrt{2}}{2}$.

If $a=0$, then $B(C_{\varphi_{a,1}})=\{1\}$ is convex. Let $0< a \leq \frac{\sqrt{2}}{2}$, we want to show that $B(C_{\varphi_{a,1}})$ is convex. First, we want to show that for $0< r <1$, $B(C_{\varphi_{a,1}})|_{re^{i\theta}}=\gamma(\theta)$ where $\gamma$ is a smooth loop, $\gamma(0)=\gamma(2\pi)$ and $\gamma'(\theta) \neq 0$ exists. From \eqref{polarform}, 
\[\mathrm{Re}(\widetilde{C}_{\varphi_{a,1}})(r,0)=\mathrm{Re}(\widetilde{C}_{\varphi_{a,1}})(r,2\pi)=\left(\dfrac{1-r^2}{1-2ar+r^2}\right)^2\left(1-ar\right)^2\text{ and }\mathrm{Re}(\widetilde{C}_{\varphi_{a,1}})(r,\pi)=\left(\dfrac{1-r^2}{1+2ar+r^2}\right)^2\left(1+ar\right)^2,\]
where $\mathrm{Re}(\widetilde{C}_{\varphi_{a,1}})(r,0) > \mathrm{Re}(\widetilde{C}_{\varphi_{a,1}})(r,\pi)$ holds. For $\theta \in [0,2\pi]$,
\begin{align*}
\widetilde{C}_{\varphi_{a,1}}(0,\theta)=1\ \ \mathrm{and}\ \ \ \ \widetilde{C}_{\varphi_{a,1}}(1,\theta)=0.
\end{align*}
For $\theta \in (0,\pi)$, we observe that $\mathrm{Im}(\widetilde{C}_{\varphi_{a,1}})(r,\theta) < 0$. Evaluating $\dfrac{\partial}{\partial\theta}\mathrm{Re}(\widetilde{C}_{\varphi_{a,1}})$, we have
\begin{align}\label{DerReBR}
\dfrac{\partial}{\partial\theta}\mathrm{Re}(\widetilde{C}_{\varphi_{a,1}})(r,\theta)=&-\dfrac{4ar\sin(\theta)(r-1)^2(r+1)^2\left[-ar^3\cos(\theta)-\dfrac{1}{2}+\left(a^2+\dfrac{1}{2}\right)r^2\right]}{(2ar\cos(\theta)-r^2-1)^3}.
\end{align}
For a small $\varepsilon >0$, $\dfrac{\partial}{\partial\theta}\mathrm{Re}(\widetilde{C}_{\varphi_{a,1}})(r,\varepsilon) < 0$. By continuity of the Berezin range on $\theta \in [0,\pi]$, we have a curve $\gamma_1(\theta) \subset \{z \in \mathbb{C}|\ \mathrm{Im}(z) \leq 0\}$ where $\gamma_1(0)$ and $ \gamma_{1}(\pi)$ are real valued, and $\gamma_1(\pi) < \gamma_{1}(0)$. Since $-\mathrm{Im}(\widetilde{C}_{\varphi_{a,1}})(r,\theta)=\mathrm{Im}(\widetilde{C}_{\varphi_{a,1}})(r,-\theta)$ and $\mathrm{Re}(\widetilde{C}_{\varphi_{a,1}})(r,\theta)=\mathrm{Re}(\widetilde{C}_{\varphi_{a,1}})(r,-\theta)$ for $\theta \in [\pi,2\pi]$, the Berezin range traces a curve $\gamma_2 \subset \{z \in \mathbb{C}|\mathrm{Im}(z) \geq 0\}$ where $\gamma_2(\pi)$ and $ \gamma_{2}(2\pi)$ are real valued, and $\gamma_{2}(\pi) < \gamma_{2}(2\pi)$. Moreover, $\gamma_1$ and $\gamma_2$ are conjugate symmetric, that is $\overline{\gamma_1}=\gamma_2$. For any $r \in (0,1)$, we have $ B(C_{\varphi_{a,1}})|_{re^{i\theta}}=\gamma_{1,r}(\theta) \cup \gamma_{2,r}(\theta)=\gamma_r(\theta)$ which shows that the restriction of the Berezin range to circles on the disk is a loop.

Next, we want to show that the Berezin range does not contain holes. From \eqref{polarform}, we have 
\[\dfrac{\partial}{\partial r}\mathrm{Re}(\widetilde{C}_{\varphi_{a,1}})(r,0)=\dfrac{8(r+1)(ar-1)(r-1)\left(a^2r^3-\dfrac{1}{4}ar^4-\dfrac{3}{2}ar^2-\dfrac{1}{4}a+r\right)}{(2ar-r^2-1)^3}.\]
We then solve $\dfrac{\partial}{\partial r}\mathrm{Re}(\widetilde{C}_{\varphi_{a,1}})(0,0) =2a>0$, indicating that the Berezin range is increasing at $r=0$. Since
\begin{align*}
\widetilde{C}_{\varphi_{a,1}}(0,0)=1\ \ \mathrm{and}\ \ \ \ \widetilde{C}_{\varphi_{a,1}}(1,0)=0.
\end{align*}
and the Berezin range is increasing at $r=0$, by the extreme value theorem, there is an absolute maximum that the Berezin range achieves at $r_0$ for $0<r_0<1$. The loop traced by $B(C_{\varphi_{a,1}})|_{re^{i\theta}}=\gamma_r(\theta)$ starts at $1$ then traverses to the right until reaching $B(C_{\varphi_{a,1}})|_{r_0e^{i\theta}}=\gamma_{r_0}(\theta)$. After reaching the absolute maximum at $r_0$, the loops change direction traversing to the left going to zero. Assuming that the Berezin range has a hole impedes the loops from traversing $0$ to $1$ continuously. Thus, $B(C_{\varphi_{a,1}})$ has no holes.

Now, we focus at the boundary of the Berezin range, $b(B(C_{\varphi_{a,1}})) \subset \{z \in \mathbb{C}|\mathrm{Im}(\mathbb{C}) \leq 0\} $ and show that it is concave up. By complex conjugate symmetry, this shows that $b(B(C_{\varphi_{a,1}})) \subset \{z \in \mathbb{C}|\mathrm{Im}(\mathbb{C}) \geq 0\} $ is concave down. Let $r$ be fixed, we want to solve the extrema of $\mathrm{Im}(\widetilde{C}_{\varphi_{a,1}})(r,\theta)$. 
We observe that, 
\[\dfrac{\partial}{\partial\theta}\mathrm{Im}(\widetilde{C}_{\varphi_{a,1}})(r,\theta)=\dfrac{2ar[-2ar^3\cos^2(\theta)+(2a^2r^2+r^2+1)\cos(\theta)+ar^3-3ar](r-1)^2(r+1)^2}{(2ar\cos(\theta)-r^2-1)^3}.\]
We take note that $-2ar^3\cos^2(\theta)+(2a^2r^2+r^2+1)\cos(\theta)+ar^3-3ar$ is quadratic in terms of $\cos(\theta)$ and when equated to zero, has a solution $\theta_1 \in (0,\pi)$ . Specifically, the solution $\theta_1$ is a critical value  dependent to $r$ such that $\mathrm{Im}(\widetilde{C}_{\varphi_{a,1}})(r,\theta_1)$ achieves a minimum. Using the quadratic formula, one can compute that 
\begin{align}\label{theta_1}
\theta_1 = \pi - \cos^{-1}\left(\dfrac{-2a^2r^2-r^2-1+\sqrt{4a^4r^4+8a^2r^6-20r^4a^2+4a^2r^2+r^4+2r^2+1}}{4ar^3}\right)
\end{align}
Since $\mathrm{Im}(\widetilde{C}_{\varphi_{a,1}})(r,0)=\mathrm{Im}(\widetilde{C}_{\varphi_{a,1}})(r,\pi)=0$ for $0<r<1$, we achieve global minimum at $\theta_1$.
Consider  $\min(\widetilde{C}_{\varphi_{a,1}}): [0,1] \rightarrow \mathbb{R}^2$ defined by 
\[\min(\widetilde{C}_{\varphi_{a,1}})(r)=(\mathrm{Re}(\widetilde{C}_{\varphi_{a,1}})(r,\theta_1),\mathrm{Im}(\widetilde{C}_{\varphi_{a,1}})(r,\theta_1)).\]
Since $\mathrm{Im}(\widetilde{C}_{\varphi_{a,1}})(r,\theta_1)$ achieves a minimum for any $r \in (0,1)$, 
\[\{\min(\widetilde{C}_{\varphi_{a,1}})(r)|\ r \in [0,1]\} = b(B(C_{\varphi_{a,1}})) \cap \{z \in \mathbb{C}|\ \mathrm{Im}(z) \leq 0\} \]
We then verify concavity using second derivative test, 
\[\dfrac{d^2\mathrm{Im}(\widetilde{C}_{\varphi_{a,1}})}{d\mathrm{Re}(\widetilde{C}_{\varphi_{a,1}})^2}=\dfrac{\frac{d}{dr}\left(\frac{d\mathrm{Im}(\widetilde{C}_{\varphi_{a,1}})/dr}{d\mathrm{Re}(\widetilde{C}_{\varphi_{a,1}})/dr}\right)}{\frac{dRe(\widetilde{C}_{\varphi_{a,1}})}{dr}} > 0.\]
Thus, the parametric curve is concave up and by the fact that $\min(\widetilde{C}_{\varphi_{a,1}})(0)=(1,0)$ and $\min(\widetilde{C}_{\varphi_{a,1}})(1)=(0,0)$, there exists $r=r_{\mathrm{min}}$ where $\mathrm{Im}(\widetilde{C}_{\varphi_{a,1}})(r_{\mathrm{min}},\theta_1)$ is the global minimum. By complex conjugate symmetry the boundary curve $b(B(C_{\varphi_{a,1}})) \cap \{z \in \mathbb{C}|\ \mathrm{Im}(z) \geq 0\}$ is concave down and $\mathrm{Im}(\widetilde{C}_{\varphi_{a,1}})$ achieves a global maximum at $r=r_{\mathrm{min}}$.

Thus, we have shown that the loops start at $1$, and as it travels to the right and travel back to the left until it reaches the origin, the peak of the loops increases in magnitude and reaches an extremum before it decreases to zero. The peaks of the loops trace a subboundary of the Berezin range which is convex. Another part of the boundary of the Berezin range is due to the loop $B(C_{\varphi_{a,1}})|_{r_1e^{i\theta}}$ where $r_1$ is the radius so that $\mathrm{Re}(\widetilde{C}_{\varphi_{a,1}})(r_1,0)$ is the global maximum. We want to observe the contour defined by $\widetilde{C}_{\varphi_{a,1}}(r_1,\theta): [0,\theta_1] \rightarrow \mathbb{R}^2 $ where $\theta_1$ is defined in \eqref{theta_1}.  Using second derivative test to verify concavity, 
\[\dfrac{d^2\mathrm{Im}(\widetilde{C}_{\varphi_{a,1}})}{d\mathrm{Re}(\widetilde{C}_{\varphi_{a,1}})^2}=\dfrac{\frac{d}{d\theta}\left(\frac{d\mathrm{Im}(\widetilde{C}_{\varphi_{a,1}})/d\theta}{d\mathrm{Re}(\widetilde{C}_{\varphi_{a,1}})/d\theta}\right)}{\frac{dRe(\widetilde{C}_{\varphi_{a,1}})}{d\theta}} > 0.\]
This shows that the parametric curve is concave up, and by complex conjugate symmetry, the curve above is concave down. With all of this information, we have shown that the Berezin range $B(C_{\varphi_{a,1}})$ is convex. By Proposition \ref{auto_on_disk}, $B(C_{\varphi_{a,1}})$ is convex for $0< |a| \leq \frac{\sqrt{2}}{2}$.  
\end{proof}
\begin{proof}[Proof of Theorem \ref{mainresultonedimension3}]
Let $\xi=i$. We want to show that $B(C_{\varphi_{a,i}})$ is not convex by showing that $\overline{B(C_{\varphi_{a,i}})}$ is not convex. We assume the contrary.  That is, let us assume that
$\overline{B(C_{\varphi_{a,i}})}$ is convex but $0<a<1$.
One can verify that $\widetilde{C}_{\varphi_{a,i}}(0)=1$ and $\widetilde{C}_{\varphi_{a,i}}(z)=0$ for $z \in \mathbb{T}$. By convexity, the interval $[0,1] \subset \overline{B(C_{\varphi_{a,i}})}$. Let $\varepsilon > 0$, and consider $z \in (1-\varepsilon)\mathbb{T}$. From Lemma \ref{ReImBerezinTransform}, define 
\[f(z)=D(z)[A(z)^2-B(z)^2]-A(z)B(z)C(z).\]
By assuming $0<|a|<\frac{\sqrt{2}}{2}$, will show there exists $w \in (1-\varepsilon)\mathbb{T}$ such that $f(w)<0$ is the global maximum.  By appealing to Proposition \ref{auto_on_disk}, we may assume $a$ is a positive real number and $0<a<\frac{\sqrt{2}}{2}$. 
Let $\varepsilon=0$ in the equation for $f$,  we have
\[f((1-\varepsilon)e^{it})|_{\varepsilon=0}=-\left(A\left((1-\varepsilon)e^{it})\right)\Big|_{\varepsilon=0}\right)^2*C\left((1-\varepsilon)e^{it}\right)\Big|_{\varepsilon=0}.\]
Now we define \[F(t)=C\left((1-\varepsilon)e^{it}\right)\Big|_{\varepsilon=0}=1-2a\cos(t)+2a^2\cos^2(t)-a^2.\]
We compute $F'(t)=2a\sin(t)\left(1-2a\cos(t)\right)=0$ to get
$t=0,\pi,2\pi$.  Additionally, if $\frac{1}{2}\leq a<\frac{\sqrt{2}}{2}$ we get another critical number $t_0$ where $\cos(t_0)=\frac{1}{2a}$.
Clearly, $F(0)=F(2\pi)>0$.  Also, $F(\pi)>0$.  So if $0<a<\frac{1}{2}$, $\displaystyle\inf_{t\in [0,2\pi]}F(t)>0$.  By the uniform continuity of $C\left((1-\varepsilon)e^{it}\right)$ in $\varepsilon$, we have that
$C\left((1-\varepsilon)e^{it}\right)\geq M>0$ for all $\varepsilon>0$ sufficiently small. 
Hence there exists $w \in (1-\varepsilon)\mathbb{T}$ such that $f(w)<0$ is the global maximum, contradicting our convexity assumption.  Now if 
$\frac{1}{2}\leq a<\frac{\sqrt{2}}{2}$, we just need to verify $F(t_0)>0$.
We substitute $t_0$ into $F(t)$ to get
$F(t_0)=1-2a\cos(t_0)+2a^2\cos^2(t_0)-a^2=\frac{1}{2}-a^2>0$ if $a<\frac{\sqrt{2}}{2}$.
This completes the proof for $0<a<\frac{\sqrt{2}}{2}$.

We need a slightly different approach for 
$\frac{\sqrt{2}}{2}\leq a<1$. Using the formula for the real and imaginary parts of 
$\widetilde{C}_{\varphi_{a,i}}$ from Lemma \ref{ReImBerezinTransform}, let us define
\[R(re^{i\theta})=C(re^{i\theta})[A(re^{i\theta})^2-B(re^{i\theta})^2]+4A(re^{i\theta})B(re^{i\theta})D(re^{i\theta})\]
and also define
\[f(re^{i\theta})=D(re^{i\theta})[A(re^{i\theta})^2-B(re^{i\theta})^2]-A(re^{i\theta})B(re^{i\theta})C(re^{i\theta}).\]
First let us fix $\theta=\frac{3\pi}{2}$ and $a\in (-1,-\frac{\sqrt{2}}{2})$.  
Then \[R(re^{i\frac{3\pi}{2}})=(1-a^2r^2)[(1-ar)^2-(r^2-ar)^2]-4(1-ar)(r^2-ar)ar.\]
Clearly, $R(0)>0$ and $R(e^{i\frac{3\pi}{2}})<0$.  So by the intermediate value theorem, there exists $d_0\in (0,1)$ so that $R\left(d_0e^{i\frac{3\pi}{2}}\right)=0$.  Furthermore, for this choice of $d_0$, one can check that
$f\left(d_0e^{i\frac{3\pi}{2}}\right)<0$.  That is,
there exists $z_0\in \mathbb{D}$ so that
$\mathrm{Re}(\widetilde{C}_{\varphi_{a,i}})(z_0)=0$ and
$\mathrm{Im}(\widetilde{C}_{\varphi_{a,i}})(z_0)<0$.  Now 
we let $\theta=\frac{\pi}{4}$ and $a\in (\frac{\sqrt{2}}{2},1)$.  By Lemma \ref{ReImBerezinTransform}, we can solve 
\[A(re^{i\frac{\pi}{4}})=1-ar\sqrt{2}=0\] and also
\[C(re^{i\frac{\pi}{4}})=A(re^{i\frac{\pi}{4}})\] to get the solution
$r_0=\frac{\sqrt{2}}{2a}$.  If $\frac{\sqrt{2}}{2}<a<1$, then
$r_0\in (0,1)$.  Furthermore,
$\mathrm{Re}(\widetilde{C}_{\varphi_{a,i}})(r_0e^{i\frac{\pi}{4}})=0$ and
$\mathrm{Im}(\widetilde{C}_{\varphi_{a,i}})(r_0e^{i\frac{\pi}{4}})>0$.  Hence if $\frac{\sqrt{2}}{2}<a<1$, then there exists $z_0,w_0\in \mathbb{D}$ so that
$\mathrm{Re}(\widetilde{C}_{\varphi_{a,i}})(z_0)=0$, $\mathrm{Im}(\widetilde{C}_{\varphi_{a,i}})(z_0)<0$,
 $\mathrm{Re}(\widetilde{C}_{\varphi_{a,i}})(w_0)=0$, and
$\mathrm{Im}(\widetilde{C}_{\varphi_{a,i}})(w_0)>0$.  If convexity of the Berezin range is assumed, then there exists $u_0\in \mathbb{D}$ so that $\widetilde{C}_{\varphi_{a,i}}(u_0)=0$.  However, this is a contradiction, since such a $u_0$ must be in the unit circle.  Now it remains to show the boundary case, namely if $a=\frac{\sqrt{2}}{2}$.

Let $a=\frac{\sqrt{2}}{2}$, by Lemma \ref{ReImBerezinTransform}, we have $\mathrm{Im}(\widetilde{C}_{\varphi_{a,i}})(0,\theta)=\mathrm{Im}(\widetilde{C}_{\varphi_{a,i}})(1,\theta)=0$. Solving for the critical values, $\dfrac{\partial}{\partial r}\mathrm{Im}(\widetilde{C}_{\varphi_{a,i}})(r,\theta)=\dfrac{\partial}{\partial \theta}\mathrm{Im}(\widetilde{C}_{\varphi_{a,i}})(r,\theta)=0$ gives $s_1(r,\theta)=(0.4442,0.8168)$, and $s_2(r,\theta)=(0.8168,1.3577)$ where $\mathrm{Im}(\widetilde{C}_{\varphi_{a,i}})(s_1)>0$ and $\mathrm{Im}(\widetilde{C}_{\varphi_{a,i}})(s_2)<0$ indicating global maximum and minimum at $s_1$ and $s_2$, respectively. 
By convexity of the Berezin range, the line $L:= \{(\mathrm{Re}(s_1),\mathrm{Im}(s_1))t\ |\ 0 \leq t \leq 1\} \subset \overline{B(C_{\varphi_{a,i}})}$. Let $\varepsilon >0$ and consider $z \in (1-\varepsilon)\mathbb{T}$, \[G(t)=C\left((1-\varepsilon)e^{it}\right)|_{\varepsilon=0}=1-\sqrt{2}\cos(t)+\dfrac{1}{2}\cos(2t).\]
We compute $G'(t)=\sin(t)\left(\sqrt{2}-2\cos(t)\right)=0$ to get $t=0,\frac{\pi}{4},\pi, \frac{7\pi}{4}$, and $2\pi$. We can verify that $G(0)=G(2\pi)>0$ and $G(\pi)>0$, and $G\left(\frac{\pi}{4}\right)=G\left(\frac{7\pi}{4}\right)=0$. We see that $C((1-\varepsilon)e^{i\frac{\pi}{4}})=C((1-\varepsilon)e^{i\frac{7\pi}{4}})=\varepsilon$, then by the uniform continuity of $C\left((1-\varepsilon)e^{it}\right)$ in $\varepsilon$, we have that
$C\left((1-\varepsilon)e^{it}\right) \geq \varepsilon > 0$ for all $\varepsilon$ sufficiently small. Hence there exists $w \in (1-\varepsilon)\mathbb{T}$ such that $f(w) < 0$ is the global maximum, contradicting our convexity assumption that $L \subset \overline{B(C_{\varphi_{a,i}})}$.
\end{proof}
By Corollary \ref{propconvex}, we have the following result.
\begin{corollary}
The Berezin range of 
$C_{\varphi_{a,-i}}$ on $A^2(\mathbb{D})$ is convex if and only if $a=0$.
\end{corollary}
\section{\texorpdfstring{The Unit Ball in $\mathbb{C}^n$}{The Unit Ball in ℂⁿ}}
Before we prove our results, we introduce some basic notation. 
\begin{definition}
    Let $\B_{n}$ denote the unit ball in $\C^{n}$. Then the Bergman kernel for $A^{2}(\B_{n})$ is given as
    \begin{equation*}
        K_{\B_{n}}(w,z) = \frac{n!}{\pi^{n}(1-\langle w,z \rangle)^{n+1}}.
    \end{equation*}
\end{definition}
\begin{definition}
    \label{1}
    Let $a \in \B_{n}$, $P_{a}$ is the orthogonal projection of $\C^{n}$ onto span$(a)$, and let $Q_{a} = I - P_{a}$ then for $a \neq 0$
    \begin{equation*}
        P_{a}(z) = \frac{\langle z,a \rangle}{\|a\|^{2}}a.
    \end{equation*}
We note that $P_{0} \equiv 0$. Let $s_{a} = (1 - \|a\|^{2})^{\frac{1}{2}}$ and define,
    \begin{equation*}
        \varphi_{a,I}(z) = \frac{a - P_{a}(z) - s_{a}Q_{a}(z)}{1 - \langle z,a \rangle }.
    \end{equation*}
    Then $\varphi_{a,I}$ is an automorphism of $\B_{n}$.
\end{definition}
\begin{theorem}\cite{Rudin80}
Every automorphism $\varphi_{a,U}$ of $\mathbb{B}_n$ is of the form 
\[\varphi_{a,U}=U\varphi_{a,I},\]
where $U$ is a unitary transformation of $\mathbb{C}^n$.
\end{theorem}
\begin{remark}\label{remarkunitary}
It is well known that unitary automorphisms act transitively on the sphere.
\end{remark}
\subsection{Automorphisms that send 0 to a} \mbox{} \\

Consider the composition operator $C_{\varphi_{a,I}}$ with symbol $\varphi_{a,I}(z)= \dfrac{a - P_{a}(z) - s_{a}Q_{a}(z)}{1 - \langle z,a \rangle }$.
\begin{theorem}\label{Th: Blaschke type symbol}
    The Berezin transform of the composition operator $C_{\varphi_{a,I}}$ is of the following form.
    \begin{equation*}
        \widetilde{C}_{\varphi_{a,I}}(z) =
        \begin{cases}
            \left[ \dfrac{(1 - \|z\|^{2})\|a\|^{2}}{\|a\|^{2}(1-2\mathrm{Re}\langle z,a \rangle) + (1-s_{a})|\langle z,a \rangle|^2 + s_{a}\|a\|^2\|z\|^2}\right]^{n+1}(1 - \langle z,a \rangle)^{n+1}, & a \neq 0 \\
            1, & a = 0
        \end{cases}
    \end{equation*}
    where $s_{a} = (1 - \|a\|^{2})^{\frac{1}{2}}$.
\end{theorem}
\begin{proof}
    Assume that $a \neq 0$, then for any $j \in \{1,2,\dots,n\},$
    \begin{align*}
        \varphi_{a,I}(z)_{j} &= \frac{a_{j} - \frac{\langle z,a \rangle}{\|a\|^{2}}a_{j} - s_{a}z_{j} + s_{a}\frac{\langle z,a \rangle}{\|a\|^{2}}a_{j}}{1 - \langle z,a \rangle} \\
        &= \frac{\|a\|^{2} - \langle z,a \rangle + s_{a}\langle z,a \rangle}{\|a\|^{2}(1 - \langle z,a \rangle)}a_{j} - \frac{s_{a}}{1 - \langle z,a \rangle}z_{j}
    \end{align*}
    Further since if $a = 0$ we have for every $j \in \{1,2,\dots,n\},$
    \begin{equation*}
        \varphi_{a,I}(z)_{j} = z_{j}.
    \end{equation*}
    Thus, for $\varphi_{a,I}(z) = (\varphi_{a,I}(z)_{1}, \varphi_{a,I}(z)_{2}, \dots ,\varphi_{a,I}(z)_{n})$
    \begin{equation*}
        \varphi_{a,I}(z)_{j} =
        \begin{cases}
            \frac{\|a\|^{2} - \langle z,a \rangle + s_{a}\langle z,a \rangle}{\|a\|^{2}(1 - \langle z,a \rangle)}a_{j} - \frac{s_{a}}{1 - \langle z,a \rangle}z_{j}, & a \neq 0 \\
            z_{j}, & a = 0
        \end{cases}
        .
    \end{equation*}
    Now we will calculate $K_{\B_{n}}(\varphi_{a,I}(z),z)$. Assume that $a \neq 0$, then
    \begin{align*}
        K_{\B_{n}}(\varphi_{a,I}(z),z) 
                                  &= \frac{n!}{\pi^{n}(1-\langle \varphi_{a,I}(z),z \rangle)^{n+1}} \\
                                  &= \frac{n!}{\pi^{n}}\left[1- \sum_{j=1}^{n}\left(\frac{\|a\|^{2} - \langle z,a \rangle + s_{a}\langle z,a \rangle}{\|a\|^{2}(1 - \langle z,a \rangle)}a_{j}\overline{z_{j}} - \frac{s_{a}}{1 - \langle z,a \rangle}|z_{j}|^{2}\right)\right]^{-(n+1)} \\
                                  &= \frac{n!}{\pi^{n}}\left[1-\left(\frac{\|a\|^{2} - \langle z,a \rangle + s_{a}\langle z,a \rangle}{\|a\|^{2}(1 - \langle z,a \rangle)}\right) \sum_{j=1}^{n}a_{j}\overline{z_{j}} + \left(\frac{s_{a}}{1 - \langle z,a \rangle}\right)\sum_{j=1}^{n}|z_{j}|^{2}\right]^{-(n+1)} \\
                                  &= \frac{n!}{\pi^{n}}\left[1-\left(\frac{\|a\|^{2} - \langle z,a \rangle + s_{a}\langle z,a \rangle}{\|a\|^{2}(1 - \langle z,a \rangle)}\right) \langle a,z \rangle + \left(\frac{s_{a}}{1 - \langle z,a \rangle}\right)\|z\|^2\right]^{-(n+1)} \\
                                  &= \frac{n!}{\pi^{n}}\left[1-\left(\frac{\overline{\langle z,a \rangle}\|a\|^{2} - |\langle z,a \rangle|^2 + s_{a}|\langle z,a \rangle|^2}{\|a\|^{2}(1 - \langle z,a \rangle)}\right) + \frac{s_{a}\|z\|^2}{1 - \langle z,a \rangle}\right]^{-(n+1)} \\
                                  &= \frac{n!}{\pi^{n}}\left[\dfrac{\|a\|^{2}(1 - \langle z,a \rangle)}{\|a\|^{2}(1 - \langle z,a \rangle)}-\left(\frac{\overline{\langle z,a \rangle}\|a\|^{2} - |\langle z,a \rangle|^2 + s_{a}|\langle z,a \rangle|^2}{\|a\|^{2}(1 - \langle z,a \rangle)}\right) + \frac{s_{a}\|z\|^2}{1 - \langle z,a \rangle}\right]^{-(n+1)} \\
                                  &= \frac{n!}{\pi^{n}}\left[\frac{\|a\|^{2}(1-2\mathrm{Re}\langle z,a \rangle) + |\langle z,a \rangle|^2 - s_{a}|\langle z,a \rangle|^2 + s_{a}\|a\|^2\|z\|^2}{\|a\|^{2}(1 - \langle z,a \rangle)}\right]^{-(n+1)} \\
                                  &= \frac{n!}{\pi^{n}}\left[\frac{\|a\|^{2}(1 - \langle z,a \rangle)}{\|a\|^{2}(1-2\mathrm{Re}\langle z,a \rangle) + |\langle z,a \rangle|^2 - s_{a}|\langle z,a \rangle|^2 + s_{a}\|a\|^2\|z\|^2}\right]^{(n+1)}
    \end{align*}
    Thus,
    \begin{equation*}
        K_{\B_{n}}(\varphi_{a,I}(z),z) = 
        \begin{cases}
            \frac{n!}{\pi^{n}}\left[\frac{\|a\|^{2}(1 - \langle z,a \rangle)}{\|a\|^{2}(1-2\mathrm{Re}\langle z,a \rangle) + |\langle z,a \rangle|^2 - s_{a}|\langle z,a \rangle|^2 + s_{a}\|a\|^2\|z\|^2}\right]^{(n+1)}, & a \neq 0 \\
            \frac{n!}{\pi^{n}(1-\|z\|^{2})^{n+1}}, & a = 0
        \end{cases}
        .
    \end{equation*}
    But this means that by using Lemma \ref{BerezinRangeComposition} we have,
        \begin{align*}
        \widetilde{C}_{\varphi_{a,I}}(z) &= 
        \begin{cases}
            \left[\frac{(1 - \|z\|^{2})\|a\|^{2}(1 - \langle z,a \rangle)}{\|a\|^{2}(1-2\mathrm{Re}\langle z,a \rangle) + |\langle z,a \rangle|^2 - s_{a}|\langle z,a \rangle|^2 + s_{a}\|a\|^2\|z\|^2}\right]^{n+1}, & a \neq 0 \\
            1, & a = 0
        \end{cases} \\
        &= 
        \begin{cases}
            \left[\frac{(1 - \|z\|^{2})\|a\|^{2}}{\|a\|^{2}(1-2\mathrm{Re}\langle z,a \rangle) + (1-s_{a})|\langle z,a \rangle|^2 + s_{a}\|a\|^2\|z\|^2}\right]^{n+1}(1 - \langle z,a \rangle)^{n+1}, & a \neq 0 \\
            1 ,& a = 0
        \end{cases}
        .
    \end{align*}
\end{proof}
\begin{theorem}\label{RealImaginary}
    The real and imaginary parts of the Berezin transform of the composition operator $C_{\varphi_{a,I}}$ are of the form,
    \begin{align*}
        \mathrm{Re}( \widetilde{C}_{\varphi_{a,I}}(z)) &=
        \begin{cases}
         \Gamma 
        \displaystyle\sum_{\substack{k=0 \\ \mathrm{k\text{ even}}}}^{n+1} {n + 1 \choose k} (-1)^{\frac{3k}{2}}[1 - \mathrm{Re}\langle z,a \rangle]^{n+1-k}  [\mathrm{Im}\langle z,a \rangle]^{k}, & a \neq 0\\
        1, & a = 0   
        \end{cases}
        \\
        \mathrm{Im}( \widetilde{C}_{\varphi_{a,I}}(z)) &=
        \begin{cases}
        \Gamma
        \displaystyle\sum_{\substack{k=0 \\ \mathrm{k\text{ odd}}}}^{n+1} {n + 1 \choose k} (-1)^{\frac{3k-1}{2}}[1 - \mathrm{Re}\langle z,a \rangle]^{n+1-k}  [\mathrm{Im}\langle z,a \rangle]^{k}, & a \neq 0\\
        0, & a = 0    
        \end{cases}     
    \end{align*}
     where
  \[\Gamma = \left[\frac{(1 - \|z\|^{2})\|a\|^{2}}{\|a\|^{2}(1-2\mathrm{Re}\langle z,a \rangle) + (1 - s_{a})|\langle z,a \rangle|^2 + s_{a}\|a\|^2\|z\|^2}\right]^{n+1}.\]
\end{theorem}
\begin{proof}
    Since the only complex valued part of $\widetilde{C}_{\varphi_{a,I}}(z)$ is $(1 - \mathrm{Re}(\langle a,z \rangle) - i\mathrm{Im}(\langle a,z \rangle))^{n+1}$. 
    Apply binomial formula to get,
    \begin{equation*}
        \widetilde{C}_{\varphi_{a,I}}(z) =
        \begin{cases}
        \Gamma \displaystyle\sum_{k=0}^{n+1} {n + 1 \choose k}[1 - \mathrm{Re}\langle z,a \rangle]^{n+1-k}  [-i\mathrm{Im}\langle z,a \rangle]^{k}, & a \neq 0\\
        1, & a = 0   
        \end{cases}
        .
    \end{equation*}
    After some simple computations we arrive at the stated equations.
\end{proof}
\begin{proposition}\label{propconj1}
The Berezin range of $C_{\varphi_{a,I}}$ on $A^2(\B_n)$ is closed under complex conjugation.
\end{proposition}
\begin{proof}
Let $z=(r_1e^{i\theta_1},\dots,r_ne^{i\theta_n})$ and nonzero $a=(\rho_1e^{i\psi_1},\dots,\rho_ne^{i\psi_n})$. We want to show that there exists $w \in \B_n$ such that 
\[\widetilde{C}_{\varphi_{a,I}}(z)= \overline{\widetilde{C}_{\varphi_{a,I}}(w)}.\]
Our candidate is $w=(r_1e^{i(2\psi_1-\theta_1)},\dots,r_ne^{i(2\psi_n-\theta_n)}) \in \B_n$. We observe
\begin{align*}
\langle w,a \rangle =& \sum_{k=1}^n\left(r_ke^{i(2\psi_k-\theta_k)}\right)\left(\rho_ke^{-i\psi_k} \right) \\
=& \sum_{k=1}^n \left(\rho_ke^{i\psi_k} \right)\left(r_ke^{-i\theta_k} \right) \\
=& \langle a,z \rangle = \overline{\langle z,a \rangle}.
\end{align*}
Equivalently, $\overline{\langle w,a \rangle}=\langle z,a \rangle$. Thus, we have
\begin{align*}
\overline{\widetilde{C}_{\varphi_{a,I}}(w)} =& \left[\frac{(1 - \|w\|^{2})\|a\|^{2}}{\|a\|^{2}(1-2\mathrm{Re}\langle w,a \rangle) + (1-s_{a})|\langle w,a \rangle|^2 + s_{a}\|a\|^2\|w\|^2}\right]^{n+1}(1 - \overline{\langle w,a \rangle})^{n+1} \\
=& \left[\frac{(1 - \|z\|^{2})\|a\|^{2}}{\|a\|^{2}(1-2\mathrm{Re}\langle z,a \rangle) + (1-s_{a})|\langle z,a \rangle|^2 + s_{a}\|a\|^2\|z\|^2}\right]^{n+1}(1 - \langle z,a \rangle)^{n+1} \\
=& \widetilde{C}_{\varphi_{a,I}}(z).
\end{align*}
\end{proof}
The following proposition is a several variables analog of Proposition \ref{auto_on_disk} with rotation replaced with action by unitary maps.  Consequently, the proof of the following proposition is very similar to the proof of Proposition \ref{auto_on_disk}.
\begin{proposition}\label{propunitary}
    Let $U$ be a unitary automorphism of $\mathbb{B}_n$.  Then the Berezin range
    $B(C_{\varphi_{Ua,I}})$ on $A^2(\mathbb{B}_n)$ is convex if and only if the Berezin range $B(C_{\varphi_{a,I}})$ on $A^2(\mathbb{B}_n)$ is convex.
\end{proposition}
\begin{proof}
Clearly, if $a=0$, then $B(C_{\varphi_{Ua,I}})=B(C_{\varphi_{a,I}})=\{1\}$ is convex.  So
let us assume $a\neq 0$.  Since $U$ is unitary, $U^*$ is also unitary.  Then by Theorem \ref{Th: Blaschke type symbol},
\begin{align*}
\widetilde{C}_{\varphi_{Ua,I}}(z)&=\left[\frac{(1 - \|z\|^{2})\|Ua\|^{2}}{\|Ua\|^{2}(1-2\mathrm{Re}\langle z,Ua \rangle) + (1-s_{Ua})|\langle z,Ua \rangle|^2 + s_{Ua}\|Ua\|^2\|z\|^2}\right]^{n+1}(1 - \langle z,Ua \rangle)^{n+1}\\
&=\left[\frac{(1 - \|U^*z\|^{2})\|a\|^{2}}{\|a\|^{2}(1-2\mathrm{Re}\langle U^*z,a \rangle) + (1-s_{a})|\langle U^*z,a \rangle|^2 + s_{a}\|a\|^2\|U^*z\|^2}\right]^{n+1}(1 - \langle U^*z,a \rangle)^{n+1}\\
&=\widetilde{C}_{\varphi_{a,I}}(U^*z).
\end{align*}
Since $U^*$ is surjective, $\widetilde{C}_{\varphi_{a,I}}(\mathbb{B}_n)=\widetilde{C}_{\varphi_{a,I}}(U^*(\mathbb{B}_n))=\widetilde{C}_{\varphi_{Ua,I}}(\mathbb{B}_n)$ as sets.  Thus we have shown the proposition.
\end{proof}
Let us now prove our main results.
\begin{proof}[Proof of Theorem \ref{Zerotopositive}]
If $a = 0$, then
\begin{equation*}
\widetilde{C}_{\varphi_{a,I}}(z) = 1, \ \forall z \in \B_{n} \implies B(C_{\varphi_{a,I}}) = \{1\}
\end{equation*}
which is clearly convex.  Let us assume that $B(C_{\varphi_{a,I}})$ is convex but that  $a\neq 0$ satisfies \[\|a\|> \tan\left(\frac{\pi}{2(n+1)}\right) .\]  By appealing to Proposition \ref{propunitary}, we may assume that $a$ has real components.  Assuming that \[\|a\|> \tan\left(\frac{\pi}{2(n+1)}\right), \] the idea is to show that there exists $c_0>0$ so that $ic_0\in B(C_{\varphi_{a,I}}) $.  By Proposition \ref{propconj1}, this will imply that
$-ic_0\in B({C}_{\varphi_{a,I}})$.  Hence assuming convexity of $B({C}_{\varphi_{a,I}})$, this will imply that $0\in B({C}_{\varphi_{a,I}})$.  This will produce a contradiction since if $\widetilde{C}_{\varphi_{a,I}}(w_0)=0$ then $\|w_0\|=1$. We let $z=ixa\|a\|^{-1}$ where $x \in (0,1)$ is to be determined.  Then $\mathrm{Re}\langle z,a\rangle=0$.
Define $H(x)=\mathrm{Re}\left(\left(1-ix\|a\|^{-1}\langle a, a\rangle\right)^{n+1}\right)$.  Then for $\Gamma$ defined as in Theorem \ref{RealImaginary},
\[\mathrm{Re}\left(\widetilde{C}_{\varphi_{a,I}}(ixa\|a\|^{-1})\right)=\Gamma H(x).\]  Since $\Gamma>0$ on $\mathbb{B}_n$, it suffices to show that there exists $x_0\in (0,1)$ so that $H(x_0)=0$, implying 
\[i\mathrm{Im}\left(\widetilde{C}_{\varphi_{a,I}}(ix_0a\|a\|^{-1})\right)=\widetilde{C}_{\varphi_{a,I}}(ix_0a\|a\|^{-1}).\]
Clearly, $H$ is continuous on $[0,1]$. Therefore we can apply the intermediate value theorem to $H(x)$. Note that $H(0)>0$.  Then let us compute the following.
\[H(x)=\left(1+x^2\|a\|^2\right)^{\frac{n+1}{2}}\mathrm{Re}\left(e^{-i(n+1) \arctan({x\|a\|})}\right)=\left(1+x^2\|a\|^2\right)^{\frac{n+1}{2}}\cos\left((n+1) \arctan({x\|a\|})\right).\]  Now if \[\frac{\pi}{2}<(n+1)\arctan(x\|a\|)<\pi\]
then $\cos\left((n+1) \arctan({x\|a\|})\right)<0$, which implies that $H(x)<0$.  Let us use the inequality \[\frac{\pi}{2}<(n+1)\arctan(x\|a\|)<\pi\] and our assumption that \[\frac{\tan\left(\frac{\pi}{2\left(n+1\right)}\right)}{\|a\|}<1\] to show that such an $x\in (0,1)$ exists.  We solve for $x$ to get the following.
\[\frac{\tan\left(\frac{\pi}{2\left(n+1\right)}\right)}{\|a\|}<x<\frac{\tan\left(\frac{\pi}{\left(n+1\right)}\right)}{\|a\|}.\]  Therefore, we can define 
\[x_1=\frac{1}{2}\left(\frac{\tan\left(\frac{\pi}{2\left(n+1\right)}\right)}{\|a\|}+\min\left\{1,\frac{\tan\left(\frac{\pi}{\left(n+1\right)}\right)}{\|a\|}\right\}\right). \]
It is clear that $x_1\in (0,1)$ by construction and our assumption that
\[\frac{\tan\left(\frac{\pi}{2\left(n+1\right)}\right)}{\|a\|}<1.\]
Furthermore, $H(x_1)<0$. Thus by the intermediate value theorem, $H(x_0)=0$ for some $x_0\in (0,x_1)$.  That is, \[ic_0=i\mathrm{Im}\left(\widetilde{C}_{\varphi_{a,I}}(ix_0a\|a\|^{-1})\right)=\widetilde{C}_{\varphi_{a,I}}(ix_0a\|a\|^{-1})\] where $c_0>0$.  So by using Proposition \ref{propconj1}, $-ic_0$ is also in the Berezin range. Therefore, by our convexity assumption, the line segment joining $-ic_0$ to $ic_0$ is also contained in the Berezin range.  Thus $0\in B({C}_{\varphi_{a,I}})$.  That is, there exists $w_0\in \mathbb{B}_n$ so that $\widetilde{C}_{\varphi_{a,I}}(w_0)=0$. This implies that $\Gamma=0$ since \[\left(1-\langle w_0,a\rangle\right)^{n+1}\neq 0.\]  This is our contradiction because $\Gamma=0$ implies $\|w_0\|=1$.  Thus if $B({C}_{\varphi_{a,I}})$ is convex, then \[\|a\|\leq \tan\left(\frac{\pi}{2(n+1)}\right). \] This proves our theorem.
\end{proof}
\subsection{Automorphisms that send 0 to -a}  \mbox{} \\
We consider the composition operator $C_{\varphi_{a,-I}}$ with automorphism symbol $\varphi_{a,-I}(z)= -\left(\frac{a-P_a(z)-s_aQ_a(z)}{1-\langle z,a\rangle}\right)$ that sends $0$ to $-a$.
\begin{theorem}\label{Th: Blaschke type symbol 0 to -a}
    The Berezin transform of the composition operator $C_{\varphi_{a,-I}}$ is of the following form.
    \begin{equation*}
        \widetilde{C}_{\varphi_{a,-I}}(z) =
        \begin{cases}
            \left[\dfrac{(1 - \|z\|^{2})\|a\|^{2}(1 - \langle z,a \rangle)}{\|a\|^{2} - (1-s_a)|\langle z,a \rangle|^2 - s_{a}\|a\|^2\|z\|^2-2i\|a\|^{2}\mathrm{Im}\langle z,a \rangle}\right]^{n+1}, & a \neq 0 \\
            1, & a = 0
        \end{cases}
    \end{equation*}
    where $s_{a} = (1 - \|a\|^{2})^{\frac{1}{2}}$.
\end{theorem}
\begin{proof}
Let $\varphi_{a,-I}(z) = (\varphi_{a,-I}(z)_{1}, \varphi_{a,-I}(z)_{2}, \dots ,\varphi_{a,-I}(z)_{n})$. Then, 
    \begin{align*}
        \varphi_{a,-I}(z)_{j} &= \frac{-a_{j} + \frac{\langle z,a \rangle}{\|a\|^{2}}a_{j} + s_{a}z_{j} - s_{a}\frac{\langle z,a \rangle}{\|a\|^{2}}a_{j}}{1 - \langle z,a \rangle} \\
        &= \frac{-\|a\|^{2} + \langle z,a \rangle - s_{a}\langle z,a \rangle}{\|a\|^{2}(1 - \langle z,a \rangle)}a_{j} + \frac{s_{a}}{1 - \langle z,a \rangle}z_{j},
    \end{align*}
for $j \in \{1,2,\dots,n\}$. If $a = 0$, we have
\begin{equation*}
    \varphi_{a,-I}(z)_{j} = z_{j}.
\end{equation*}
Thus, for $\varphi_{a,-I}(z) = (\varphi_{a,-I}(z)_{1}, \varphi_{a,-I}(z)_{2}, \dots ,\varphi_{a,-I}(z)_{n})$ 
    \begin{equation*}
        \varphi_{a,-I}(z)_{j} =
        \begin{cases}
            \frac{-\|a\|^{2} + \langle z,a \rangle - s_{a}\langle z,a \rangle}{\|a\|^{2}(1 - \langle z,a \rangle)}a_{j} + \frac{s_{a}}{1 - \langle z,a \rangle}z_{j}, & a \neq 0 \\
            z_{j}, & a = 0
        \end{cases}.
    \end{equation*}
    
Now we will calculate $K_{\B_{n}}(\varphi_{a,-I}(z),z)$. Assume that $a \neq 0$, then by a similar method as used in the proof of Theorem \ref{Th: Blaschke type symbol} we get that,
    \begin{align*}
        K_{\B_{n}}(\varphi_{a,-I}(z),z) = \frac{n!}{\pi^{n}}\left[\frac{\|a\|^{2}(1 - \langle z,a \rangle)}{\|a\|^{2}(1-2i\mathrm{Im}\langle z,a \rangle) - |\langle z,a \rangle|^2 + s_{a}|\langle z,a \rangle|^2 - s_{a}\|a\|^2\|z\|^2}\right]^{(n+1)}
\end{align*}
Assume that $a = 0$, we have 
\[K_{\B_{n}}(\varphi_{a,-I}(z),z) = \frac{n!}{\pi^{n}(1-\|z\|^{2})^{n+1}}.\]
By Lemma \ref{BerezinRangeComposition}, we have
     \begin{align*}
        \widetilde{C}_{\varphi_{a,-I}}(z) &= 
        \begin{cases}
            \left[\frac{(1 - \|z\|^{2})\|a\|^{2}(1 - \langle z,a \rangle)}{\|a\|^{2}(1-2i\mathrm{Im}\langle z,a \rangle) - |\langle z,a \rangle|^2 + s_{a}|\langle z,a \rangle|^2 - s_{a}\|a\|^2\|z\|^2}\right]^{n+1}, & a \neq 0 \\
            1, & a = 0
        \end{cases}
        \\
        &= 
        \begin{cases}
            \left[\frac{(1 - \|z\|^{2})\|a\|^{2}(1 - \langle z,a \rangle)}{\|a\|^{2} - (1-s_a)|\langle z,a \rangle|^2 - s_{a}\|a\|^2\|z\|^2-2i\|a\|^{2}\mathrm{Im}\langle z,a \rangle}\right]^{n+1}, & a \neq 0 \\
            1, & a = 0
        \end{cases}
        .
    \end{align*}
\end{proof}
\begin{theorem}
    The real and imaginary parts of the Berezin transform of the composition operator $C_{\varphi_{a,-I}}$ are of the form,
    \begin{align*}
        \mathrm{Re}( \widetilde{C}_{\varphi_{a,-I}}(z)) &= \eta \left[
        \sum_{\substack{k=0 \\ \mathrm{k \ even}}}^{n+1} {n + 1 \choose k} \alpha^{n+1-k} (-1)^{\frac{k}{2}} \beta^{k}\right] \\
        \mathrm{Im}( \widetilde{C}_{\varphi_{a,-I}}(z)) &= \eta \left[
        \sum_{\substack{k=0 \\ \mathrm{k \ odd}}}^{n+1} {n + 1 \choose k} \alpha^{n+1-k} (-1)^{\frac{k-1}{2}} \beta^{k}\right],     
    \end{align*}
     where
    \[\eta=\left[\frac{\|a\|^{2}(1 - \|z\|^{2})}{\big{|} \lambda - 2i\|a\|^{2}\mathrm{Im}\langle z,a \rangle\big{|}^2}\right]^{n+1},\]
    
    \[\alpha=\lambda(1 - \mathrm{Re}\langle z,a\rangle)  + 2\|a\|^{2}(\mathrm{Im}\langle z,a \rangle)^2,\]
    
    \[\beta=\mathrm{Im}\langle z,a \rangle\big{[}2\|a\|^2(1 - \mathrm{Re}\langle z,a\rangle)-\lambda \big{]}, \text{and}\]

    \[\lambda=\|a\|^{2} - (1-s_a)|\langle z,a \rangle|^2 - s_{a}\|a\|^2\|z\|^2.\]
\end{theorem}
\begin{proof}
We have
     \begin{align*}
        \widetilde{C}_{\varphi_{a,-I}}(z) &= 
        \begin{cases}
            \left[\frac{(1 - \|z\|^{2})\|a\|^{2}(1 - \langle z,a \rangle)}{\|a\|^{2} - (1-s_a)|\langle z,a \rangle|^2 - s_{a}\|a\|^2\|z\|^2-2i\|a\|^{2}\mathrm{Im}\langle z,a \rangle}\right]^{n+1}, & a \neq 0 \\
            1, & a = 0
        \end{cases} \\
        &= 
         \begin{cases}
            \left[\frac{\left[(1 - \|z\|^{2})\|a\|^{2}(1 - \langle z,a \rangle)\right]\left[\|a\|^{2} - (1-s_a)|\langle z,a \rangle|^2 - s_{a}\|a\|^2\|z\|^2+2i\|a\|^{2}\mathrm{Im}\langle z,a \rangle \right]}{\big{|}\|a\|^{2} - (1-s_a)|\langle z,a \rangle|^2 - s_{a}\|a\|^2\|z\|^2-2i\|a\|^{2}\mathrm{Im}\langle z,a \rangle\big{|}^2}\right]^{n+1}, & a \neq 0 \\
            1, & a = 0
        \end{cases} \\
        &= 
         \begin{cases}
          \left\{\eta\big{[}(1 - \mathrm{Re}\langle z,a\rangle) - i\mathrm{Im}\langle z,a\rangle \big{]}\big{[}\lambda+2i\|a\|^{2}\mathrm{Im}\langle z,a \rangle \big{]}\right\}^{n+1}, & a \neq 0 \\
            1, & a = 0
        \end{cases} \\
        &= 
         \begin{cases}
          \eta\left\{\big{[}\lambda(1 - \mathrm{Re}\langle z,a\rangle)  + 2\|a\|^{2}(\mathrm{Im}\langle z,a \rangle)^2\big{]} + i\mathrm{Im}\langle z,a \rangle\big{[}2\|a\|^2(1 - \mathrm{Re}\langle z,a\rangle)-\lambda \big{]} \right\}^{n+1}, & a \neq 0 \\
            1, & a = 0
        \end{cases}
    \end{align*}
    where      
    \[\eta=\left[\frac{\|a\|^{2}(1 - \|z\|^{2})}{\big{|} \lambda-2i\|a\|^{2}\mathrm{Im}\langle z,a \rangle\big{|}^2}\right]^{n+1},\text{ and }
    \lambda=\|a\|^{2} - (1-s_a)|\langle z,a \rangle|^2 - s_{a}\|a\|^2\|z\|^2.\]
Using binomial theorem, we get the desired result.
\end{proof}\begin{proposition}\label{propconju}
The Berezin range of $C_{\varphi_{a,-I}}$ on $A^2(\B_n)$ is closed under complex conjugation.
\end{proposition}
\begin{proof}
Let $z=(r_1e^{i\theta_1},\dots,r_ne^{i\theta_n})$ and nonzero $a=(\rho_1e^{i\psi_1},\dots,\rho_ne^{i\psi_n})$. We want to show that there exists $w \in \B_n$ such that 
\[\widetilde{C}_{\varphi_{a,-I}}(z)= \overline{\widetilde{C}_{\varphi_{a,-I}}(w)}.\]
Our candidate is $w=(r_1e^{i(2\psi_1-\theta_1)},\dots,r_ne^{i(2\psi_n-\theta_n)}) \in \B_n$. We observe
\begin{align*}
\langle w,a \rangle =& \sum_{k=1}^n\left(r_ke^{i(2\psi_k-\theta_k)}\right)\left(\rho_ke^{-i\psi_k} \right) \\
=& \sum_{k=1}^n \left(\rho_ke^{i\psi_k} \right)\left(r_ke^{-i\theta_k} \right) \\
=& \langle a,z \rangle \\
=& \overline{\langle z,a \rangle}.
\end{align*}
Furthermore, $\overline{\langle w,a \rangle}=\langle z,a \rangle$ and $\mathrm{Im} \overline{\langle w,a \rangle}= -\mathrm{Im} \langle z,a \rangle $. Thus, we have
\begin{align*}
\overline{\widetilde{C}_{\varphi_{a,-I}}(w)} =& \overline{  \left[\frac{(1 - \|w\|^{2})\|a\|^{2}(1 - \langle w,a \rangle)}{\|a\|^{2} - (1-s_a)|\langle w,a \rangle|^2 - s_{a}\|a\|^2\|w\|^2-2i\|a\|^{2}\mathrm{Im}\langle w,a \rangle}\right]^{n+1}} \\
=&   \left[\frac{(1 - \|w\|^{2})\|a\|^{2}(1 - \overline{\langle w,a \rangle})}{\|a\|^{2} - (1-s_a)|\langle w,a \rangle|^2 - s_{a}\|a\|^2\|z\|^2+2i\|a\|^{2}\mathrm{Im}\overline{\langle w,a \rangle}}\right]^{n+1} \\
=&   \left[\frac{(1 - \|z\|^{2})\|a\|^{2}(1 - \langle z,a \rangle)}{\|a\|^{2} - (1-s_a)|\langle z,a \rangle|^2 - s_{a}\|a\|^2\|z\|^2 - 2i\|a\|^{2}\mathrm{Im}\langle z,a \rangle}\right]^{n+1} \\
=& \widetilde{C}_{\varphi_{a,-I}}(z).
\end{align*}
\end{proof}
\begin{proposition}\label{propunitary2}
    Let $U$ be a unitary automorphism of $\mathbb{B}_n$.  Then the Berezin range
    $B(C_{\varphi_{Ua,-I}})$ on $A^2(\mathbb{B}_n)$ is convex if and only if the Berezin range $B(C_{\varphi_{a,-I}})$ on $A^2(\mathbb{B}_n)$ is convex.
\end{proposition} \noindent
The proof of the above proposition follows the proof of Proposition \ref{propunitary}.

Let us prove our main result.
\begin{proof}[Proof of Theorem \ref{Zerotonegative}]
If $a = 0$, then $\widetilde{C}_{\varphi_{a,-I}}(z) = 1$ for $z \in \B_{n}$. This implies $B(C_{\varphi_{a,-I}}) = \{1\}$ which is convex. 

Now suppose $B(C_{\varphi_{a,-I}})$ on $A^{2}(\B_{n})$ is convex. By Proposition \ref{propunitary2}, convexity of $B(C_{\varphi_{a,-I}})$ is equivalent to convexity of $B(C_{\varphi_{Ua,-I}})$ for any unitary automorphism $U$.  Now it is well known that the group of unitary automorphisms acts transitively on spheres.  Therefore we may assume that $a$ is real valued since we can construct unitary $U$ to that $Ua$ is real valued.  

We let $h(c)=ic\frac{a}{\|a\|}$ for $c\in [0,1]$. The idea is to compute the real part of $\widetilde{C}_{{\varphi}_{a,-I}}(h(c))$ to show that there exists $c_0\in (0,1)$ so that
$\mathrm{Re}(\widetilde{C}_{{\varphi}_{a,-I}}(h(c_0)))=0$.Assuming convexity of the Berezin range, this will produce a contradiction.

We assume that $a\neq 0$ is real valued and the Berezin range is convex.  Using Theorem \ref{Th: Blaschke type symbol 0 to -a}, one can write
\begin{equation*} 
\widetilde{C}_{{\varphi}_{a,-I}}(h(c))=\left(\frac{1}{D(c)}(1-c^2)\|a\|^2\left(\left(\|a\|^2(1-c^2)+2\|a\|^4c^2\right)+i\left(2\|a\|^3c+\|a\|^3c (c^2-1)\right)\right)\right)^{n+1}
\end{equation*}
where $D(c)\geq M>0$ for all $c\in [-1,1]$. Now we compute the argument of $\widetilde{C}_{{\varphi}_{a,-I}}(h(c))$.
We have
\[\theta(c)=\mathrm{Arg}\left(\widetilde{C}_{{\varphi}_{a,-I}}(h(c))\right)=(n+1)\arctan\left(\frac{\|a\|c(1+c^2)}{2\|a\|^2c^2+(1-c^2)}\right).\]
Since $\theta$ is continuous on $[0,1]$, $\theta(0)=0$, and $\theta(1)>\frac{\pi}{2}$, there exists $c_0\in (0,1)$ so that
$\theta(c_0)=\frac{\pi}{2}$ by the intermediate value theorem.  That is,
$\mathrm{Re}(\widetilde{C}_{{\varphi}_{a,-I}}(h(c_0)))=0$.  Hence
$\widetilde{C}_{{\varphi}_{a,-I}}(h(c_0))\neq 0$ lies on the imaginary axis.  Hence by Proposition \ref{propconju}, 
$-\widetilde{C}_{{\varphi}_{a,-I}}(h(c_0))$ is also in the Berezin range.  Now assuming convexity of the Berezin range, we have $0$ is in the Berezin range.  This is a contradiction since if
$\widetilde{C}_{{\varphi}_{a,-I}}(z_0)=0$ then $\|z_0\|=1$.
Hence $a=0$.
\end{proof}

\subsection*{Acknowledgment}
We wish to thank S\"onmez {\c{S}}ahuto\u{g}lu for helpful comments on a preliminary version of this manuscript. 
\bibliographystyle{amsalpha}
\bibliography{BibFile}

\end{document}